\documentclass[11pt]{amsart} 

\usepackage{geometry}             
\usepackage{float}
\usepackage{fullpage,xcolor}
\usepackage[all]{xy}
\usepackage{epsfig}
\usepackage{amsfonts}
\usepackage{graphicx}
\usepackage{amssymb}
\usepackage{amsmath}
\usepackage{amsthm}
\usepackage{amsopn}
\usepackage{epstopdf}
\usepackage[font=small]{caption}
\usepackage{url}
\usepackage{hyperref}

\usepackage{subfig} 
\newcommand{\ZZ}{{\mathbb Z}}

\newcommand{\RR}{{\mathbb R}}

\newcommand{\lk}{\operatorname{lk}}

\renewcommand{\phi}{\varphi}

\newtheorem{theorem}{Theorem} 
\newtheorem{lemma}[theorem]{Lemma}
\newtheorem{proposition}[theorem]{Proposition}

\theoremstyle{definition}

\theoremstyle{remark}

\newtheorem*{acknowledgments}{Acknowledgments}

\title{Classification of virtual string links up to cobordism}

\author[R. Gaudreau]{Robin Gaudreau}
\address{Mathematics, University of Toronto, Toronto, Ontario}
\email{gaudreau@tutamail.com}

\keywords{Virtual links, string links, cobordism, concordance, welded knots.}

\date{\today}  

\begin{document}

\maketitle

\begin{abstract} Cobordism of virtual string links on $n$ strands is a combinatorial generalization of link cobordism. There exists a bijection between virtual string links up to cobordisms and elements of the group $\ZZ^{n(n-1)}$. This paper also shows that virtual string links up to unwelded equivalence are classified by those groups. Finally, the related theory of welded string link cobordism is defined herein and shown to be trivial for string links with one component (Theorem \ref{new}).
\end{abstract}

\hspace{2em}

\section*{Introduction}
Virtual knot theory, as understood from \cite{Ka1999}, is a combinatorial extension of classical knot theory. When picturing oriented knots as diagrams in the plane, crossings are vertices of a planar, oriented, tetra-valent graph with a cyclic orientation of the edges and a distinguished over-crossing pair. Removing the planarity requirement on such graphs yields virtual knot diagrams, whose equivalence classes up to the appropriate (generalized) Reidemeister moves are called virtual knots. Similarly, by understanding a classical string link as an equivalence class of diagrams, one defines a virtual string link. The goal of this paper is to relate three generalizations of concepts from classical knot theory through the following result:

\begin{theorem}
Let $L_1$ and $L_2$ be virtual string link diagrams. Then, the following are equivalent:

\begin{enumerate}
\item $L_1$ is cobordant to $L_2$,
\item $L_1$ is unwelded equivalent to $L_2$,
\item The pairwise virtual linking number of the components of $L_1$ equal those of $L_2$.
\end{enumerate}
\end{theorem}

The first classification, string link cobordism, is a generalization of the notion of virtual link cobordism introduced by Carter, Kamada, and Saito in \cite{CKS2002}. It relies on an interpretation of virtually knotted objects as curves in thickened surfaces, but yields the same theory as the one that is exposed below.

The second classification has been studied under many other names, notably as \emph{fused isotopy} and the equivalence between statements (2) and (3) is a generalization of Theorem 2 of \cite{FK2007} and of Theorem 8 in \cite{Ok2005}. Unfortunately, unwelded equivalence lacks an intrinsic topological interpretation.

The paper is structured as follows: relevant definitions are given in Section \ref{background}, results and topological notions needed to prove the main theorem appear in Section \ref{known}, followed by its proof in Section \ref{proof}. Finally, Section \ref{welded} is contains partial results on welded string links up to concordance. 

\section{Vocabulary} \label{background}

From now on, fix $n\ge 1$ to be an integer, and let $I=[0,1]$ denote the closed unit interval. 

\subsection{Virtual string links}

Classical string links were defined in \cite{HL1990} as a \emph{self-concordance} of $n$ points in $D^2\times I$, where $D^2$ is the closed unit disk in the plane. This abstract and succinct definition contains all the details needed to understand these objects, but it does not allow a straightforward generalization to virtual string links. 

Following the approach to virtual knot theory from \cite{Ka1999}, let a virtual string link diagram be a diagram consisting of $n$ smooth curves, oriented from $(\frac{i}{n+1}, 0)$ to $(\frac{i}{n+1}, 1)$, with $i=1, 2, \ldots n$, such that singularities are at most a finite number of transverse double points, decorated in one of the ways depicted in Figure \ref{crossings}. The classical Reidemeister moves, as shown in Figure \ref{GD} can be applied to string link diagrams and generate the expected equivalence classes. 

Therefore, virtual string links, to be generalizations of virtual pure braids need to be defined combinatorially. While one could do this process using any knot presentation, the following will only use planar and Gauss diagrams. \emph{Virtual string links} are then the equivalence class generated by such a diagram, up to the extended Reidemeister moves from \cite{Ka1999} and planar isotopies.

Given such a planar diagram, one can create its associated Gauss diagram by drawing the $n$ intervals and connecting the pre-images of a classical crossing by an arrow oriented from the overcrossing component to the undercrossing one, decorated by signs using the convention shown in Figure \ref{crossings}. The writhe function of a crossing $c$, $w(c)$, takes value $+1$ or $-1$ if $c$ is positive or negative respectively. The writhe is not defined for virtual crossings.

Alternatively, a virtual string link Gauss diagram can be constructed abstractly, by drawing a finite number of signed oriented chords with distinct endpoints on the interior of $n$ intervals. As with virtual knots and links, it is immediate that any such Gauss diagram can be realized as a virtual string link planar diagram. Gauss diagrams admit their own version of Reidemeister moves, which are the same for classical and virtual $SL_1$ since Reidemeister moves which involve virtual crossings leave the Gauss diagram unchanged. 

In this paper, ``Reidemeister moves'' is used to mean simultaneously the classical moves on planar diagrams, their extended version, and the analogous moves on Gauss diagrams. Figure \ref{GD} shows the equivalence between the two approaches, and therefore a virtual string link can be defined strictly from the Gauss diagrams.

\begin{figure}[ht] 
\centering
\includegraphics[scale=0.7]{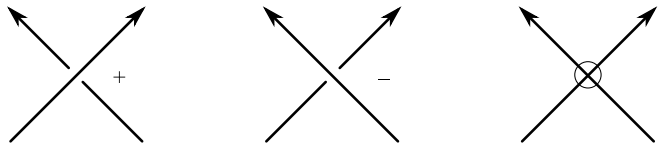}  
\caption{Positive, negative, and virtual crossings in planar diagrams.}  \label{crossings}
\end{figure}

\subsection{Notation}

Following \cite{BBD2015}, the set of classical string links on $n$ strands is denoted $uSL_n$. Its virtual extension is $vSL_n$, while the welded version, $vSL_n/(f1)$ is $wSL_n$. Finally, $vSL_n/(f1,f2)=:uwSL_n$ are unwelded string links on $n$ strands. The moves $f1$ and $f2$ are discussed in more details in Section \ref{forb}.

\hfill 

\begin{figure}[ht] 
\centering
\includegraphics[scale=0.6]{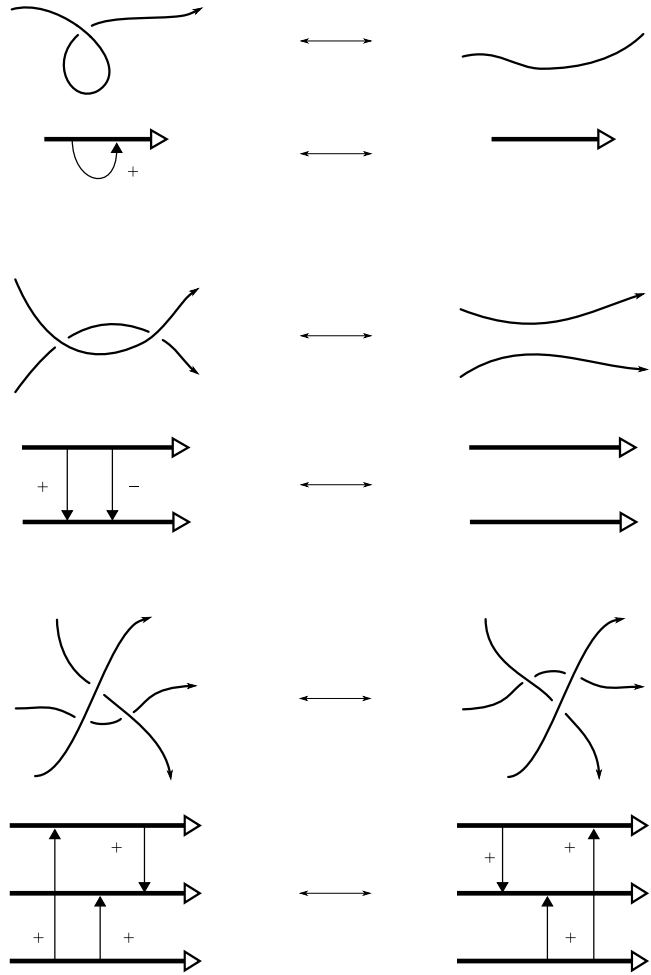}  
\caption{Reidemeister moves on planar and Gauss diagrams.}  \label{GD}
\end{figure}

Each of those sets is closed under an ordered, binary operation, the connected sum. It is written using the $\#$ operator, following the notation for $uSL_1$, which corresponds to classical long knots, and to carry on the analogy, diagrams are drawn such that the strands connect vertical intervals from left to right. Given two string link diagrams on $n$ strands, $D_1$ and $D_2$, their connected sum $D_1 \# D_2$ is represented a diagram obtained by connected the end of the $i$th strand of $D_1$ to the beginning of the $i$th strand of $D_2$. For Gauss diagrams, the connected sum is also represented by concatenation, as seen in Figure \ref{inverse}, this time joining the pre-images of the strands together.

Because moves can be applied to each part of $D_1 \# D_2$ independently, the result of a connected sum is independent of the choice of diagrams. Moreover, the operation is associative, thus makes the sets of string links into monoids.

\begin{figure}[htbp] 
\centering
\includegraphics[scale=0.7]{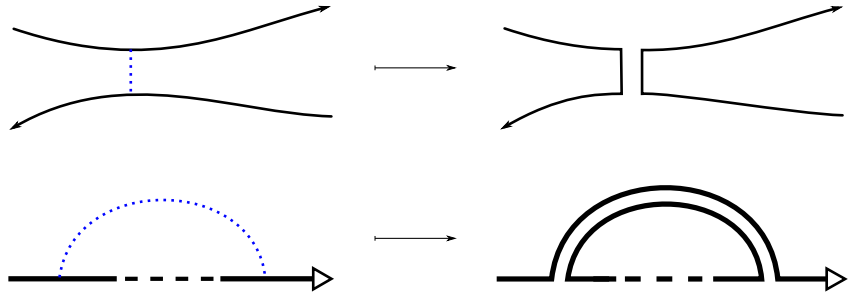}  
\caption{The saddle move on a planar and on a Gauss diagram.} \label{saddle}
\end{figure}

\subsection{Virtual cobordims}

A cobordism between two virtual knot diagrams $K_0$ and $K_1$ is a finite sequence of Reidemeister moves, births and deaths of unknotted components, and oriented saddle moves, as pictured in Figure \ref{saddle}. Diagrammatic cobordism are generalized to string links from \cite{CKS2002}, with the added restriction that the abstract surface with corners described by the cobordism of an $n$-component virtual string link must have precisely $n$ connected components. This corresponds to the topological restriction that is imposed on link cobordisms. As with other cases of cobordism, the genus can be computed by using the formula $$(s_i-b_i+d_i)/2,$$ where $s_i$ is the number of saddle moves in the cobordism that involve the $i$th component, $b_i$ the number of birthed unknots that get saddled to it and $d_i$ the number of deaths related to the component. 

\hfill 

\begin{figure}[htbp] \centering
\includegraphics[scale=0.75]{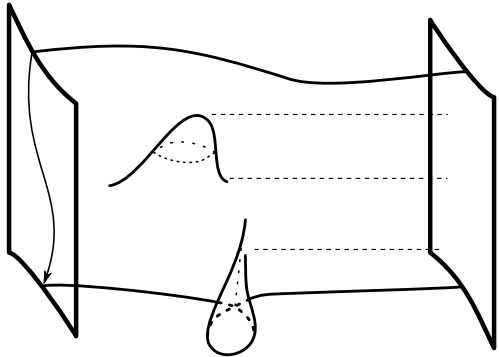}  
\caption{Concordance between a standard long unknot and one with a kink.} \label{RM1cob}
\end{figure}

Given a virtual string link diagram $D$ on $n$ strands, the cobordism class it generates is $\mathcal B(D)$, and the set of all such classes is $vSL_n\mathcal B$. Similarly, the restriction to classical diagrams is $uSL_n\mathcal B$. 

For classical knots, a cobordism between $K_0, K_1\subset \RR^3$ is called a concordance if it is realized by an annulus $S^1 \times [0,1]\subset \RR^3\times I$ where its boundary component $S^1\times \{i\}$ represents $K_i$.  For long knots, the cobording surface is $\RR\times [0,1]$, and a simple truncated example is in Figure \ref{RM1cob}. The set of diagrams that are concordant to some classical string link $D$ is $\mathcal C(D)$, an element $uSL_n\mathcal C$, the $n$-strand classical string link concordance group (with the inverse of an planar diagram being its vertical mirror image). Using the abstract definition of genus above, a \emph{concordance} between two virtual string link diagrams on $n$ strands consists of a series of extended Reidemeister moves, births, deaths and saddles, such that the genus of the cobordism on each component is 0. The quotient by concordance of $vSL_n$ is denoted $vSL_n\mathcal C$, and called the $n$-strand virtual string link concordance group (see Proposition \ref{invprop}). As with cobordism, all quotients of $vSL_n$ can be factored by concordance equivalence, and there are many questions about the maps between those groups. It is known from \cite{BN2017} that $uSL_1\mathcal C$ embeds in $vSL_1\mathcal C$, but it is an open problem whether this continues to hold for $n>1$. i.e. is the natural map $uSL_n\mathcal C \to vSL_n\mathcal C$ one-to-one?

While round classical knots up to concordance form a group, round virtual knots do not have a well-defined connected sum, hence the appropriate virtual concordance group uses long virtual knots and agrees with $vSL_1\mathcal C$. This motivates the study of the problem above.

\subsection{Forbidden moves} \label{forb}
On planar diagrams, forbidden moves are the tempting operations that appear similar to a third Reidemeister move and would allow a strand to slide either over ($f1$) or under ($f2$) a virtual crossing. On Gauss diagrams, the difference between those operations and the other moves is more evident as the forbidden moves allow certain arrow endpoints to commute without compensating for it elsewhere in the link, as depicted in Figure \ref{forbidden}. 

\hfill

\begin{figure}[htbp] \centering
\includegraphics[scale=0.6]{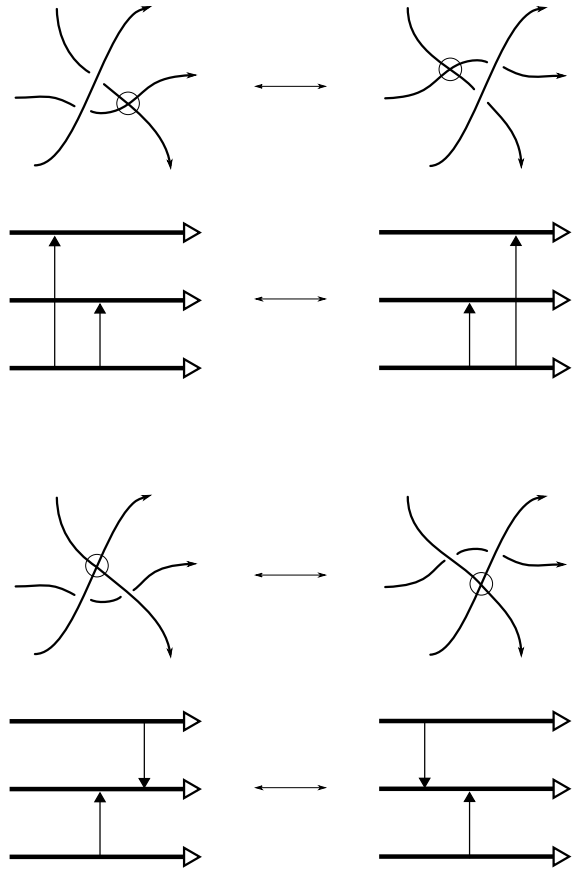}  
\caption{Forbidden moves of planar and Gauss diagrams.} \label{forbidden}
\end{figure}

The first forbidden move has appeared, and been allowed, in the literature long before virtual knot theory was ever popular. Keeping with the notation introduced in \cite{FRR1997}, the objects defined up to Reidemeister moves and the first forbidden move are called \emph{welded}. That paper is focused on welded braids, and proves that the welded pure braid groups are not trivial, and distinct from the classical pure braid groups. Further allowing the second forbidden move yields \emph{unwelded} objects. In particular, all knots are trivial as unwelded knots, as shown in \cite{Ne2001} and references therein. The main theorem of this paper is proved in Section \ref{sam} following Nelson's approach.

Let $vSL_n\mathcal B$ denote the monoid whose elements are equivalence classes of virtual string link diagrams on $n$ components up to cobordisms and whose operation is concatenation. The monoids $wSL_n\mathcal B$ and $uwSL_n\mathcal B$ are defined similarly by allowing one and both forbidden moves respectively. 

\section{Fundamental results} \label{known}

The classical linking between two components of a classical link was first defined as an integral over the paths of a representative of the link and it admits combinatorial formulas that compute it from a planar or Gauss diagram. Using the un-normalization version, $$\operatorname{ulk}(L_{(1)},L_{(2)})=\sum_{c\in L_{(1)}\cap L_{(2)}}\operatorname{w}(c),$$ where $L_{(i)}$ are components of a link $L$, $c$ a crossing, an $\operatorname{w}(c)$ the sign of $c$.

If $L$ is a classical knot, then this ``usual'' linking number is even, and often normalized by multiplication by a factor of $\frac{1}{2}$. For virtual link, the symmetry that this relies on needs not hold and the \emph{ordered} linking numbers are different. Let $$\operatorname{lk}(L_{(1)},L_{(2)})=\sum_{c : L_{(1)} \to L_{(2)}} \operatorname{w}(c)$$ be the linking number of $L_{(1)}$ over $L_{(2)}$, that is, the sum of the writhes of the crossings where $L_{(1)}$ goes over $L_{(2)}$. The notation $L_{(1)}\to L_{(2)}$ reflects that the arrows that are counted in the Gauss diagram point from $L_{(1)}$ to $L_{(2)}$. Then, $\operatorname{ulk}(L_{(1)},L_{(2)})= \operatorname{ulk}(L_{(1)},L_{(2)})=\operatorname{lk}(L_{(1)},L_{(2)})+ \operatorname{lk}(L_{(2)},L_{(1)})$. 

These definitions can be used verbatim for components of virtual string links.

\begin{figure}[htbp] 
\centering\includegraphics[scale=0.7]{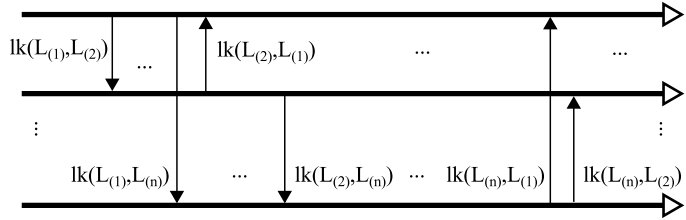}  
\caption{The standard form of an unwelded string link.} \label{standard}
\end{figure} 

\begin{lemma} \label{firstlem}
The linking numbers between components of a virtual link or virtual string link are invariant under the forbidden moves and cobordisms.
\end{lemma}

\begin{proof}
Let $L$ and $L'$ differ by a single forbidden move. Let $L_{(1)}$ and $L_{(2)}$ be components of $L$, and $L'_{(i)}$, $i=1,2$, be the corresponding components in $L'$. Since forbidden moves change neither the number nor the sign of arrows between any two components, $\lk(L_{(1)}, L_{(2)})= \lk(L'_{(1)},L'_{(2)})$. For cobordisms, first notice that the corresponding claim also holds for the first and third Reidemeister moves. For the second Reidemeister move, assume that $L'$ is obtained from $L$ by canceling a pair of arrows from $L_{(1)}$ to $L_{(2)}$. Then, those arrows contribute $+1$ and $-1$ respectively to $\lk(L_{(1)}, L_{(2)})$ and therefore $\lk(L_{(1)}, L_{(2)})= \lk(L'_{(1)}, L'_{(2)})$. Finally, the restrictions on the death, birth, and saddle moves make it so that  the order of the endpoints of arrows can be changed, but the component on which they lie is preserved. Therefore, pairwise linking numbers are invariants of $uwSL_n$ and $vSL_n\mathcal B$.
\end{proof}

\hfill

\begin{figure}[htbp] \centering
\includegraphics[scale=0.75]{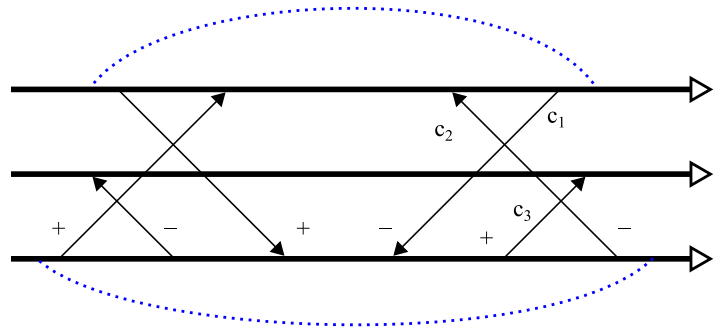}  
\caption{Connected sum of a virtual string link diagram and its concordance inverse. Dashed curves show the saddle moves needed to trivialize it.} \label{inverse}
\end{figure}

Proposition 4.9 in \cite{CKS2002} shows that cobordism classes of virtual links are completely classified by pairwise virtual linking numbers. Any virtual string link can be mapped to an oriented virtual link with the same number of components by connecting the endpoints of each strand together without creating new crossings. This operation is called the \emph{closure}. It immediately follows from the main theorem that the closure on $vSL_n\mathcal B$ is an injection onto cobordism classes of virtual links with $n$ components.

\begin{lemma} \label{cobcom}
Forbidden moves can be realized by cobordisms.
\end{lemma}

\begin{proof}
Consider chords $c$ and $d$ in a Gauss diagram such that they have adjacent endpoints. Then, by using two saddle moves, these endpoints can be first exiled to some small closed component, and then reintroduced to their original location with the opposite order. Such a cobordism realizes both the first and the second forbidden move, and any other move which commutes arrows. 
\end{proof}

Additivity of the linking numbers under connected sum is immediate from its definition. This is where the third classification from Theorem 1 comes from. Let $\operatorname{LK_n}: uwSL_n\to \ZZ^{n(n-1)}$ be the map that takes an unwelded string link $L$ to the ordered list of its linking numbers, $(\lk(L_{(1)}, L_{(2)}), \lk(L_{(1)}, L_{(3)}), \ldots, \lk(L_{(n-1)}, L_{(n)})$.  In particular, there is standard form for any unwelded string link which displays exactly the crossings which contribute to $\lk(L_{(i)},L_{(j)})$ in lexicographic order on $ij$. This is illustrated in Figure \ref{standard}, with the sign of the crossing replaced by the signed number of parallel arrows with that sign. Conversely, given any list $w$ of $n(n-1)$ integers, there is a unique unwelded string $W$ such that $LK_n(W)=w$, using that standard form. Combining Lemmas \ref{firstlem} and \ref{cobcom}, this discussion also holds for $\operatorname{LK_n}: vSL_n \mathcal B \to \ZZ^{n(n-1)}$.

\begin{proposition}
The monoids $vSL_n\mathcal B$, $wSL_n\mathcal B$, and $uwSL_n\mathcal B$ are isomorphic for all $n\ge 1$, and $vSL_1\mathcal B$ has exactly one element.
\end{proposition}

\begin{proof}
First notice that $vSL_n\mathcal B \twoheadrightarrow wSL_n\mathcal B\twoheadrightarrow uwSL_n\mathcal B $, since each monoid is obtained from the previous one by allowing one more move. By Lemma \ref{cobcom}, $ uwSL_n\mathcal B \twoheadrightarrow vSL_n\mathcal B$, and thus all those maps are isomorphisms. 

For the second part of the statement, it suffices to show that any arrow in a one-component string link Gauss diagram $D$ can be erased using a cobordism and classical  Reidemeister moves. 

This is done by creating a saddle parallel to the crossing such that its head and foot are adjacent. Then, it can be removed using a single $RM1$ and the saddled off component can be reconnected to its original component of the link with a saddle more. Since $D$ has only one component, every arrow can be canceled that way, and $D$ is cobordant to the empty Gauss diagram on one long component. Thus, $vSL_1\mathcal B$ is isomorphic to $\{1\}$. \end{proof}

\section{Proof of the theorem} \label{proof}
The main step is to show that the endpoints of any two adjacent chords on the Gauss diagram can be commuted using unwelded equivalences or using cobordisms. The proof could equivalently be illustrated with planar diagrams, but the simplicity of the standard form can be lost in the sea of virtual crossings that is required to realize it. 

\subsection*{Standard form with cobordisms}
Let $D$ be a virtual string link diagram. By Lemma \ref{cobcom} the endpoints on each strand commute with each other. Thus, self-crossings can be isolated and removed, while the rest of the diagram can be organized to be in standard form, by canceling parallel arrows with opposite signs as needed.

\subsection*{Standard form with forbidden moves}  \label{sam}

This argument is a generalization of the proof that forbidden moves unknot virtual knots as it appears in \cite{Ne2001}. The first and second forbidden moves on planar diagrams admit many orientations which give all possible choices of signs to the pairs of chords depicted in Figure \ref{given}. Thus, any two adjacent arrowheads or arrowfeet on a component of a string link Gauss diagram can commute.

\begin{figure}[htbp] \centering
\centering\includegraphics[scale=0.75]{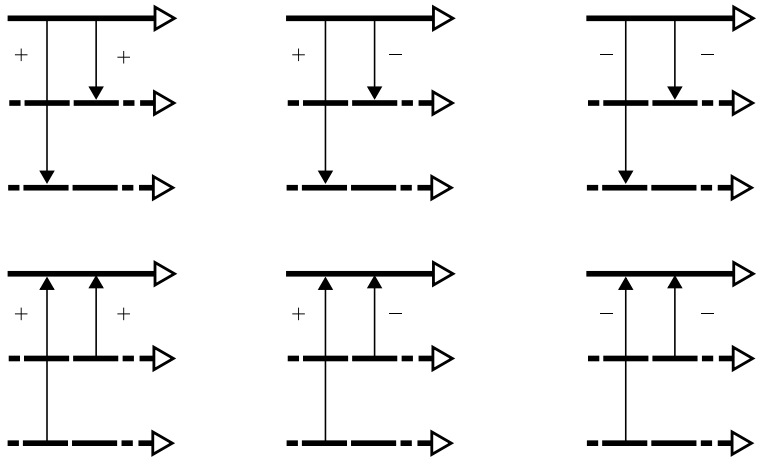}  
\caption{Pairs of crossings which commute with forbidden moves.} \label{given}
\end{figure}

There are four different choices of signs that can occur in this situation. Two of them are depicted in Figure \ref{forbidden-commute}. The other cases can be obtained from these by applying various symmetries to the diagrams and changes of orientation of the strings. 
\hfill

\begin{figure}[htbp]\centering
\centering\includegraphics[scale=0.75]{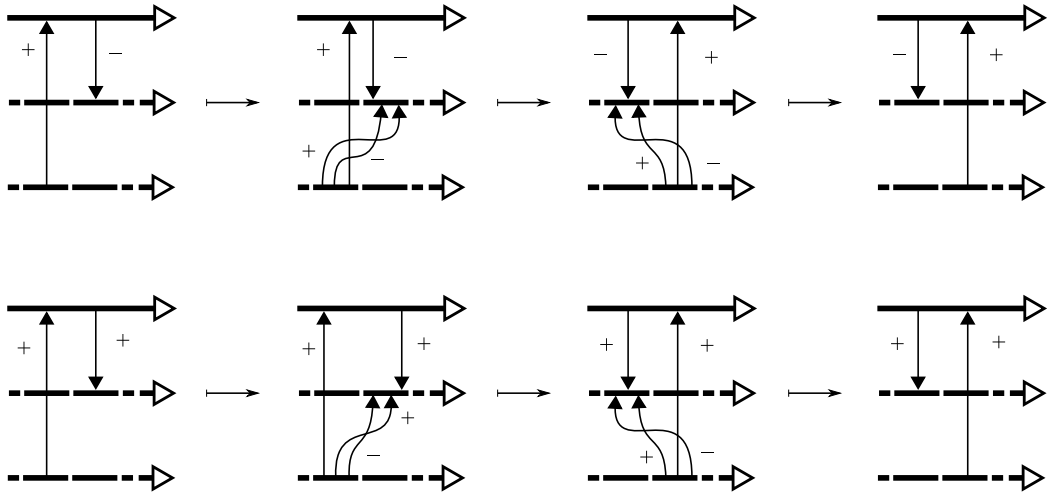}  
\caption{Commuting crossings using both forbidden moves.}    \label{forbidden-commute}
\end{figure}

It follows that the order of arrows on each component is irrelevant to the unwelded string link represented by a Gauss diagram. As with cobordisms, placing any self-crossing as an isolated crossing allows them to be canceled and the rest of the link can be put in standard form, which is uniquely determined by the $n(n-1)$ pairwise virtual linking numbers.

\section{Welded knot concordance} \label{welded}
 
As an attempt to reach a midpoint between unwelded equivalence and cobordism, say that two virtual string link diagrams are \emph{welded concordant} if one can be obtained from the other by a sequence of generalized Reidemeister moves, first forbidden moves, and genus 0 cobordisms. The welded moves are allowed to happen at any point of the cobordism.

The Tube map was defined on virtual knot diagrams by Satoh in \cite{Sa2000}, and gives a topological setting in which to interpret the first forbidden move, which is then more accurately called the \emph{overcrossings commute} move, by mapping a planar knot diagram to a ribbon knotted torus in four dimensional space. Consider the Tube of each diagram appearing in a concordance movie between welded knots. The birth and death of unknotted components correspond respectively to creating and filling a ribbon (un)knotted torus. A proposed geometric realization of the saddle move is seen in Figure \ref{saddletube}. The Tube map of welded string links is defined in \cite{ABMW2017} while a concordance theory for ribbon knotted surfaces which agrees with welded concordance has yet to be studied.

\hfill

\begin{figure}[htbp] \centering
\centering\includegraphics[scale=0.9]{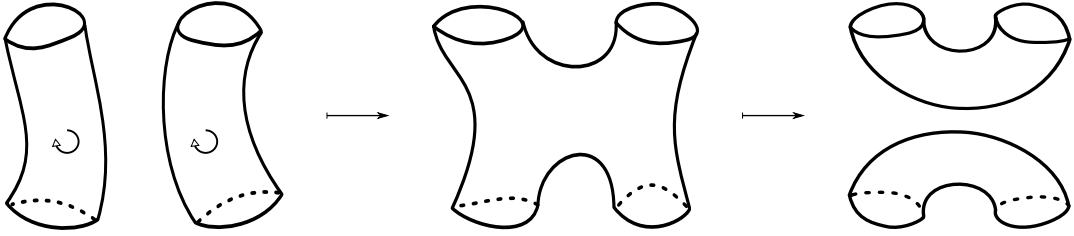}  
\caption{The saddle move realized on a ribbon knotted surface.} \label{saddletube}
\end{figure}

In the case of one component string links, Theorem \ref{new} shows that allowing the first forbidden move trivializes the virtual knot concordance group, which is surprising, considering the well-known results that classical knots inject into welded knots.

\begin{theorem}\label{new}
Any long welded knot is concordant to the unknot. 
\end{theorem}

\begin{proof}
Again, this is shown by putting an arbitrary Gauss diagram $D_0$ with $k$ crossings into the standard form using the allowed moves. 

Let $D_1$ be the diagram obtained by adding an isolated crossing of opposite sign, pointing near the arrowhead of every crossing of $D_0$. This is shown in Figure \ref{newfig}. Then, using saddle moves, each pair of arrowheads can be isolated to its own closed component. The resulting diagram is $D_2$. Since the long component of $D_2$ contains only overcrossings and this is a welded link, they can be commuted such that the pairs of crossings have adjacent feet. This is diagram $D_3$ of Figure \ref{newfig}.

\hfill

\begin{figure}[htbp] \centering
\centering\includegraphics[scale=0.7]{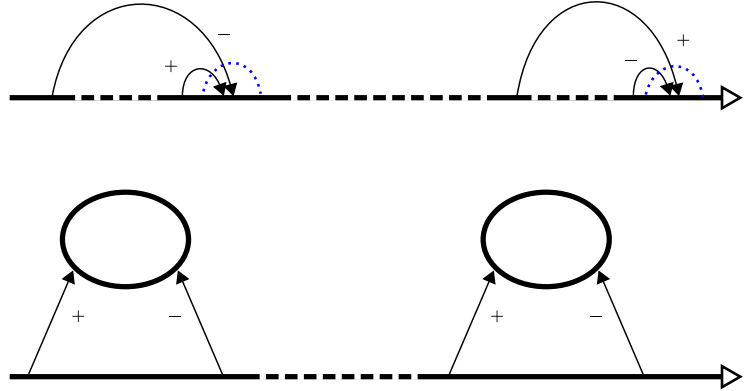}  
\caption{Gauss diagram where the long component contains only undercrossings, and the $n$ closed components each have one positive and one negative overcrossing.} \label{newfig}
\end{figure}
 The second Reidemeister move allows to cancel each of the pairs, and finally, deaths delete the closed components. Since there were $k$ saddle moves and the same number of deaths, this is a genus 0 cobordism to the unknot. 
\end{proof}

For completeness, let's mention that as for classical knots, the concordance inverse of a welded string link is obtained by taking mirror images. Proposition \ref{invprop} should be self-evident and it is presented here using Gauss diagram language. 

\begin{proposition} \label{invprop} Let $S$ be a string link Gauss diagram with $n$ components.
Let $-S$ denote the diagram obtained by changing the direction of each core component and the sign of each arrow. Then, $S\#-S=-S\#S=U_n\in wSL_n\mathcal C$. 
\end{proposition}

\begin{proof}
Since $-(-S)=S$, it suffices to prove the second equality. Enumerate the chords of $S$ as $c_1, c_2, \ldots, c_k$ such that the first crossing in the diagram is from $c_1$, the one after that from $c_2$, and so on. Denote by $-c_i$ the mirror image of $c_i$ in $-S$. Then, the innermost pair of crossings is $(-c_1,c_1)$. Using a saddle move which connects the arcs on the outside of the far endpoints of $\pm c_1$ to each other, that pair of crossings can be canceled using a second Reidemeister move. Figure \ref{inverse} shows how crossings pair up with their inverses in the Gauss diagram.

Repeat this as needed (at most $k-1$ times) creating round components, and removing them with deaths as needed until the diagram is empty. 
\end{proof}

As a corollary, $vSL_n\mathcal C$ is also a group for any $n$.

\begin{acknowledgments} 
This work would not have been possible without the feedback and encouragement from Dror Bar-Natan and Hans Boden. The result in this paper was first presented at the Knots in Washington, and the author is grateful to the organizers for creating such a motivating and creative atmosphere. This research was supported by funding from the National Science Research Council of Canada. 
\end{acknowledgments}


\end{document}